\documentclass[11pt]{article}
\usepackage[margin=1in]{geometry}
\usepackage{amsmath,amssymb,amsthm,mathtools}
\usepackage[T1]{fontenc}
\usepackage{lmodern}
\usepackage{microtype}
\usepackage{hyperref}
\usepackage{enumitem}

\hypersetup{colorlinks=true,linkcolor=blue,citecolor=blue,urlcolor=blue}

\newtheorem{theorem}{Theorem}
\newtheorem{lemma}[theorem]{Lemma}

\newtheorem{remark}[theorem]{Remark}

\newtheorem{claim}[theorem]{Claim}

\newcommand{\AB}{\operatorname{AB}}

\newcommand{\twr}{\operatorname{twr}}

\title{Above and below}
\author{%
  Wenchong Chen\thanks{Department of Mathematics, Nankai University, Tianjin, China and Department of Mathematics, Emory University, Atlanta, USA. Email: \texttt{wenchong.chen@emory.edu}.}
  \and
  Cosmin Pohoata\thanks{Department of Mathematics, Emory University, Atlanta, GA 30322, USA. Email: \texttt{cosmin.pohoata@emory.edu}.}
}
\date{}

\begin{document}
\maketitle

\begin{abstract}
We study a family of above-below Ramsey functions $\operatorname{AB}^{(d)}(k)$ defined for sequences of points in $\mathbb R^d$ whose projections to $\mathbb R^{d-1}$ have cyclic order type. The case $d=3$ is the above-below function $\AB(k)$ that was first introduced by Pohoata and Zakharov in their work on the Erd\H{o}s--Szekeres problem in $\mathbb{R}^{3}$. We prove the sharp estimate
\[
\AB(k)=2^{2^{\Theta(k)}},
\]
and, more generally, show that $\operatorname{AB}^{(d)}(k)$ is closely related to the higher-order cup-cap function of Eli\'a\v{s} and Matou\v{s}ek and the monotone Ramsey numbers of Balko.
\end{abstract}

\section{Introduction}

In a seminal paper \cite{ES35}, Erd\H{o}s and Szekeres initiated one of the central threads of combinatorial geometry by proving that every sufficiently large set of points in the plane in general position contains a large subset in convex position. More generally, for $d\ge 2$, let $ES_d(n)$ denote the smallest integer such that every set of at least $ES_d(n)$ points in $\mathbb R^d$ in general position contains $n$ points in convex position. The existence of $ES_d(n)$ follows from Ramsey's theorem, but the asymptotic behavior of these functions remained mysterious until recent years. 

In the plane, it has been long known since 1935 that $ES_2(n)$ is exponential in $n$, and in a celebrated paper from about ten years ago Suk \cite{Suk17} established that $ES_2(n)=2^{n+o(n)}$. See also \cite{HMPT20} for the best known quantitative bound, as well as a generalization to the pseudoline setting. In higher dimensions, the situation is dramatically different. Answering a question of Erd\H{o}s (see for instance \cite{Bloom}), Pohoata and Zakharov established that $ES_d(n)=2^{o(n)}$ for every fixed $d\ge 3$, by showing that already in $\mathbb R^3$ substantially fewer points are needed to force a convex polytope on $n$ vertices than are required in the plane \cite{PZ}.

One of the ingredients in the proof from \cite{PZ} is a Ramsey-type statement about points in space with convexly ordered planar projections. Let $x_1,\dots,x_N\in\mathbb R^3$ be points in general position whose orthogonal projections onto the first two coordinates $\pi(x_1),\dots,\pi(x_N)\in\mathbb R^2$ are the consecutive vertices of a convex polygon. If $i<i'<j<j'$, then the projected segments $\pi(x_i)\pi(x_j)$ and $\pi(x_{i'})\pi(x_{j'})$ cross in exactly one point. We say that the segment $x_i x_j$ is \emph{above} $x_{i'}x_{j'}$ if, over this crossing point, the point of $x_i x_j$ has larger third coordinate than the corresponding point of $x_{i'}x_{j'}$. Otherwise, we say that $x_i x_j$ is \emph{below} $x_{i'}x_{j'}$.

For $k\ge 4$, let $\AB(k)$ be the smallest integer $N$ such that every such sequence $x_1,\dots,x_N\in\mathbb R^3$ contains indices $1\le i_1<\cdots<i_k\le N$ for which either
\[
x_{i_r}x_{i_t}\text{ is above }x_{i_s}x_{i_u}
\qquad\text{for all }1\le r<s<t<u\le k,
\]
or
\[
x_{i_r}x_{i_t}\text{ is below }x_{i_s}x_{i_u}
\qquad\text{for all }1\le r<s<t<u\le k.
\]
In \cite{PZ}, this function is bounded by a straightforward application of Ramsey's theorem for $4$-uniform Ramsey numbers: $\AB(k)\le R_4(k,k)$. This was sufficient for the qualitative subexponential bound on $ES_3(n)$, but the problem of determining the asymptotic value of $\AB(k)$ as $k \to \infty$ remained an intriguing open question (one which the authors of \cite{PZ} also highlighted).

The main goal of this note is to solve this problem, by establishing the following sharp asymptotic estimate. 

\begin{theorem}\label{thm:main}
There exist absolute constants $c,C>0$ such that for all $k\ge 4$,
\[
2^{2^{ck}}\le \AB(k)\le 2^{2^{Ck}}.
\]
Equivalently,
\[
\AB(k)=2^{2^{\Theta(k)}}.
\]
\end{theorem}

Theorem~\ref{thm:main} comes from a new connection that we will establish with the work of Eli\'a\v{s} and Matou\v{s}ek \cite{EM}. In their setting, one starts with a planar sequence of points $p_i=(t_i,h_i)$, where $t_1<\cdots<t_N$, and two-colors each ordered $4$-tuple $i<j<k<\ell$ according to the sign of the expression
\begin{equation}\label{eq:third-divdiff-intro}
\Delta_3(p_i,p_j,p_k,p_\ell)
=
\sum_{a\in\{i,j,k,\ell\}} \frac{h_a}{\prod_{b\in\{i,j,k,\ell\}\setminus\{a\}}(t_a-t_b)}.
\end{equation}
This is the so-called third divided difference of $p_i,p_j,p_k,p_{\ell}$, which we will review more closely in Section 2. A subsequence of points is called third-order monotone if all of its $4$-tuples have the same color, and we denote by $\operatorname{EM}^{(3)}(k)$ the smallest integer $N$ such that every planar sequence $p_1,\dots,p_N$, with $p_i=(t_i,h_i)$, and $t_1<\cdots<t_N$, in $3$-general position contains a third-order monotone subsequence of size $k$. In \cite{EM}, Eli\'a\v{s} and Matou\v{s}ek proved that this function is doubly exponential in $k$. Their construction turns out to be the main input behind the lower bound in Theorem~\ref{thm:main}.

The main idea is roughly as follows. If we restrict to points of the form
\[
x_1=(t_1,t_1^2,h_1),\ldots,x_N=(t_N,t_N^2,h_N)\in\mathbb R^3
\]
in the definition of $\AB(k)$, then the projections $\pi(x_i)$ lie on the parabola $y=x^2$. For a quadruple $i<j<k<\ell$, a direct calculation shows that the above--below comparison between the two crossing lifted diagonals $x_i x_k$ and $x_jx_\ell$ is dictated by the sign of a single determinant. A direct calculation then shows that this determinant equals a positive Vandermonde factor times the third divided difference $\Delta_3(p_i,p_j,p_k,p_\ell)$. Consequently, up to a global interchange of the two colors, the above--below coloring in the parabola model is exactly the third-order divided-difference coloring from \eqref{eq:third-divdiff-intro}. 

Before we make things more precise, let us introduce a natural higher-dimensional version of the above-below problem, for which we can establish the connection with the higher-order cup-cap functions of Eli\'a\v{s} and Matou\v{s}ek in much greater generality. 

\bigskip

\noindent\textbf{Higher-dimensional above and below.} Let $d\ge 2$, and write points of $\mathbb R^d$ as $x_i=(z_i,h_i)$, with $z_i\in\mathbb R^{d-1}$ and $h_i\in\mathbb R$. Let
\[
\pi:\mathbb R^d\to\mathbb R^{d-1},
\qquad
\pi(z,h)=z
\]
be the projection onto the first $d-1$ coordinates. We say that a sequence $x_1,\dots,x_N\in\mathbb R^d$ has
\emph{cyclic projections} if the projected sequence
$z_i=\pi(x_i)\in\mathbb R^{d-1}$ has the same order type as points $\gamma(t_i)=(t_i,t_i^2,\dots,t_i^{d-1})$, $t_1<\cdots<t_N$ on the moment curve. For points on the moment curve, it is known that the primitive Radon partitions of the cyclic polytope $C(n,d-1)$ are exactly the alternating partitions along the moment curve (see for example, the work of Breen \cite{Breen73}). In particular, every ordered $(d+1)$-tuple $z_{i_0},z_{i_1},\dots,z_{i_d}$, with $i_0<i_1<\cdots<i_d$ has Radon partition $\{z_{i_j}:j\ \mathrm{even}\}$ versus $\{z_{i_j}:j\ \mathrm{odd}\}$, up to swapping the two parts. Since Radon partitions of minimally dependent subsets are determined by the order type\footnote{This follows from the uniqueness (up to swapping sides) of the Radon partition in Radon's theorem \cite{Radon} for $d+1$ points in general position in $\mathbb{R}^{d-1}$. Indeed, for such a $(d+1)$-tuple in $\mathbb R^{d-1}$, the signs of the coefficients in its affine dependence are determined by the signs of the $d\times d$ minors obtained by deleting one point at a time. These signs are part of the order type. Thus configurations with the same order type have the same Radon partitions on all minimally dependent subsets. See also \cite{Ziegler} for more background on cyclic polytopes.}, the same statement holds for any projected sequence with cyclic order type.

For such a tuple, let $\rho\in\mathbb R^{d-1}$ be the Radon point. After a normalization, this means that there exist positive coefficients $\alpha_j,\beta_j$, such that $\sum_{j\ \mathrm{even}}\alpha_j=1$ and $\sum_{j\ \mathrm{odd}}\beta_j=1$, and such that
\[
\rho
=
\sum_{j\ \mathrm{even}}\alpha_j z_{i_j}
=
\sum_{j\ \mathrm{odd}}\beta_j z_{i_j}.
\]
Writing $x_i=(z_i,h_i)$, define the two lifted heights over $\rho$ by
\[
H_{\mathrm{even}}
=
\sum_{j\ \mathrm{even}}\alpha_j h_{i_j},
\qquad
H_{\mathrm{odd}}
=
\sum_{j\ \mathrm{odd}}\beta_j h_{i_j}.
\]
The $d$-dimensional above--below color of the tuple $i_0<\cdots<i_d$ records whether $H_{\mathrm{even}}>H_{\mathrm{odd}}$, or the reverse.

For $d\ge 2$ and $k\ge d+1$, let $\operatorname{AB}^{(d)}(k)$ be the smallest integer $N$ such that every sequence $x_1,\dots,x_N\in\mathbb R^d$ in general position with cyclic projections contains indices $1\le i_1<\cdots<i_k\le N$ for which all $(d+1)$-tuples receive the same $d$-dimensional above--below color. For $d=3$, note that we indeed have $\operatorname{AB}^{(3)}(k)=\operatorname{AB}(k)$: four projected points in convex position always have Radon partition $\{i_0,i_2\}$ versus $\{i_1,i_3\}$, so the above--below color simply compares the two lifted crossing diagonals.

We also consider the moment-curve restricted version
$\operatorname{AB}^{(d)}_{\mathrm{mc}}(k)$, where one only allows sequences of the
form $x_i=(t_i,t_i^2,\dots,t_i^{d-1},h_i)$, where
$t_1<\cdots<t_N$. Let $\operatorname{EM}^{(d)}(k)$ denote the $d$-th order
Erd\H{o}s--Szekeres function of Eli\'a\v{s} and Matou\v{s}ek, namely the smallest
$N$ forcing a $k$-term subsequence on which all $d$-th divided differences have
the same sign. Finally, let $R^{\mathrm{trans}}_r(k)$ denote the Ramsey number
for transitive two-colorings of ordered $r$-tuples, as introduced by Fox, Pach,
Sudakov, and Suk \cite[Section 6]{FPSS}, and let
$\overline R_{\mathrm{mon}}(k;r)$ denote the corresponding Ramsey number for
$r$-monotone colorings in the sense of Balko \cite{Balko}. We will review the definitions of $\operatorname{EM}^{(d)}(k)$,
$R^{\mathrm{trans}}_r(k)$, and $\overline R_{\mathrm{mon}}(k;r)$ in detail in the next section.

The following theorem describes the general connection between the above-below functions, the work of Eli\'a\v{s} and Matou\v{s}ek, and Balko's monotone Ramsey
numbers.

\begin{theorem}\label{thm:higher-dimensional}
For every fixed $d\ge 2$,
\[
\operatorname{EM}^{(d)}(k)
=
\operatorname{AB}^{(d)}_{\mathrm{mc}}(k)
\le
\operatorname{AB}^{(d)}(k)
\le
\overline R_{\mathrm{mon}}(k;d+1).
\]
\end{theorem}

Theorem~\ref{thm:main} is therefore the first nontrivial case of a much broader hierarchy. Indeed, when $d=3$, Theorem~\ref{thm:higher-dimensional} and the lower bound $\operatorname{EM}^{(3)}(k)=2^{2^{\Omega(k)}}$ from
\cite[Theorem 1.5]{EM} give the lower bound in Theorem~\ref{thm:main}. For the upper bound in Theorem~\ref{thm:main}, it would already be enough to know that the above-below coloring is transitive and to use the estimate $R^{\mathrm{trans}}_4(k)\le 2^{2^{O(k)}}$ from \cite[Theorem 1.4]{EM}. However, the last inequality from Theorem~\ref{thm:higher-dimensional} proves a slightly more precise structural statement: the above-below coloring is in fact monotone, which is a much more rigid property. In particular, the inequality $\operatorname{EM}^{(d)}(k) \le \overline R_{\mathrm{mon}}(k;d+1)$ alone (which is trivially implied by Theorem \ref{thm:higher-dimensional}) already significantly refines the estimate $\operatorname{EM}^{(d)}(k)\le R_{d+1}^{\mathrm{trans}}(k)$ of Eli\'a\v{s} and Matou\v{s}ek from \cite{EM}. We will discuss the old connections between these parameters in Section~2, and then prove Theorem~\ref{thm:higher-dimensional} in Section~3.

\section{Preliminaries}

In this section we recall the higher-order cup-cap framework of Eli\'a\v{s} and Matou\v{s}ek \cite{EM}. Throughout, $[N]=\{1,2,\dots,N\}$ is equipped with its natural order. We define the standard tower function recursively as follows:
\[
\twr_1(x)=x,
\qquad
\twr_{r+1}(x)=2^{\twr_r(x)}.
\]

\smallskip

\noindent\textbf{Transitive colorings.} We first review the concept of transitive colorings on ordered tuples. Let $r\ge 2$. A two-coloring $\chi:\binom{[N]}{r}\to\{+,-\}$
is called \emph{transitive} if, for every $i_1<i_2<\cdots<i_{r+1}$, whenever the two consecutive $r$-tuples $\{i_1,\dots,i_r\}$ and $\{i_2,\dots,i_{r+1}\}$ have the same color, then every $r$-subset of $\{i_1,\dots,i_{r+1}\}$ has this same color. Let $R_r^{\mathrm{trans}}(n)$ be the smallest $N$ such that every transitive two-coloring of $\binom{[N]}r$ contains a subset of size $n$ where all the $r$-tuples are monochromatic.

A basic example to keep in mind is the ordinary cups-vs-caps coloring of triples in the plane. Suppose $p_1=(t_1,h_1),\ldots,p_N=(t_N,h_N) \in \mathbb{R}^{2}$ be such that $ t_1<\cdots<t_N$ and color a triple $i<j<k$ according to whether it is a cup or a cap; or, equivalently, according to the sign of
\[
\Delta_2(p_i,p_j,p_k)
=
\frac{\frac{h_k-h_j}{t_k-t_j}-\frac{h_j-h_i}{t_j-t_i}}{t_k-t_i}.
\]
This triple-coloring is transitive. Indeed, if $i<j<k<\ell$ and both consecutive triples $(i,j,k)$ and $(j,k,\ell)$ are cups, then the two adjacent slopes satisfy $\operatorname{slope}(p_ip_j)<\operatorname{slope}(p_jp_k)<\operatorname{slope}(p_kp_\ell)$. Every longer secant slope is a weighted average of the slopes of the shorter segments it spans. Hence
\[
\operatorname{slope}(p_ip_j)<\operatorname{slope}(p_jp_\ell)
\quad\text{and}\quad
\operatorname{slope}(p_ip_k)<\operatorname{slope}(p_kp_\ell),
\]
so the remaining triples $(i,j,\ell)$ and $(i,k,\ell)$ are also cups. The same argument applies if the two consecutive triples are caps. Thus the cups-vs-caps coloring is transitive. Consequently, the Ramsey function  $R_3^{\mathrm{trans}}(n)$ controls the classical cup-cap function of Erd\H{o}s-Szekeres. The key twist is that $R_3^{\mathrm{trans}}(n)$ is in fact comparable in size. In particular, it is also exponential (and, moreover, equal with the cup-cap function from \cite{ES35}). More generally, in \cite[Theorem 1.4]{EM}, Eli\'a\v{s}--Matou\v{s}ek proved for every fixed $r\ge 3$,
\begin{equation} \label{EMblah}
R_r^{\mathrm{trans}}(n)\le \twr_{r-1}(O(n)).
\end{equation}

For every $r \geq 3$, the Ramsey function $R_r^{\mathrm{trans}}(n)$ also controls an appropriate generalization of the cup-cap function. 

\bigskip

\noindent\textbf{Monotone colorings.} We now recall a more restrictive class of
ordered colorings, introduced by Balko \cite{Balko}. Let $r\ge 2$. A two-coloring $\chi:\binom{[N]}{r}\to\{+,-\}$
is called \emph{$r$-monotone} if, for every $i_0<i_1<\cdots<i_r$, the sequence of colors of the $r$-subsets of $\{i_0,\ldots,i_r\}$, taken in
lexicographic order, changes sign at most once. Equivalently, on every $(r+1)$-tuple, the color sequence is either constant or consists of a block of one color followed by a block of the other color.

Let $\overline R_{\mathrm{mon}}(k;r)$ be the least integer $N$ such that every
$r$-monotone coloring of $\binom{[N]}{r}$ contains a $k$-element subset all of
whose $r$-subsets have the same color. Since every $r$-monotone coloring is
transitive, we have $\overline R_{\mathrm{mon}}(k;r)\le R_r^{\mathrm{trans}}(k)$. Indeed, on an ordered $(r+1)$-tuple, the first and last $r$-subsets in
lexicographic order are precisely the two consecutive $r$-tuples. If they have
the same color, then the one-switch condition forces the entire sequence to be
constant.

Balko proved that, for every fixed $r\ge 3$,
\begin{equation} \label{Balko}
\overline R_{\mathrm{mon}}(k;r)=\twr_{r-1}(\Theta(k)).
\end{equation}
This gives the same tower height as the transitive Ramsey upper bound, but for a
smaller and more structured class of colorings. The one-switch condition is also
closely related to the language of consistent sets in the higher Bruhat orders:
the set of $+$-colored $r$-tuples meets every packet of $r$-subsets of an
$(r+1)$-set in an initial or terminal segment. This viewpoint goes back to the work of Manin and Schechtman on higher Bruhat orders and was developed further
by Ziegler in his study of cyclic hyperplane arrangements
\cite{ManinSchechtman,ZieglerHBO}.

\bigskip

\noindent\textbf{Higher-order monotonicity.} Let $P=\left\{p_1,\dots,p_N\right\}$ be a set of points with $p_i=(t_i,h_i)\in\mathbb R^2$, and $t_1<\cdots<t_N$. For an ordered tuple $p_{i_0},\dots,p_{i_d}$, with $i_0<\cdots<i_d$, we write $\Delta_q(p_{i_0},\dots,p_{i_d})$
for the $d$-th divided difference of the values $h_i$. It is defined recursively by $\Delta_0(p_i):=h_i$, and, for $d\ge 1$,
\[
\Delta_d(p_{i_0},\dots,p_{i_d})
:=
\frac{
\Delta_{d-1}(p_{i_1},\dots,p_{i_d})
-
\Delta_{d-1}(p_{i_0},\dots,p_{i_{d-1}})
}
{t_{i_d}-t_{i_0}}.
\]
Equivalently,
\begin{equation}\label{eq:Delta-closed}
\Delta_d(p_{i_0},\dots,p_{i_d})
=
\sum_{j=0}^{d}
\frac{h_{i_j}}
{\prod_{r\ne j}(t_{i_j}-t_{i_r})}.
\end{equation}
Formula~\eqref{eq:Delta-closed} shows that $\Delta_d$ is symmetric in the $d+1$ points. As we will see later in the proof of Theorem \ref{thm:higher-dimensional}, a useful way to remember $\Delta_d$ is through interpolation: if $f$ is the unique polynomial of degree at most $d$ passing through $p_{i_0},\dots,p_{i_d}$, then $\Delta_d(p_{i_0},\dots,p_{i_d})$ is the coefficient of $x^d$ in $f$. In particular, the tuple lies on the graph of a polynomial of degree at most $d-1$ if and only if its $d$-th divided difference vanishes.

Following \cite{EM}, we say that $P \subset \mathbb{R}^{2}$ is in \emph{$d$-general position} if no $d+1$ points of $P$ lie on the graph of a polynomial of degree at most $d-1$, equivalently if every $(d+1)$-tuple has nonzero $d$-th divided difference. A subsequence is called \emph{$d$-th order monotone} if all of its $(d+1)$-tuples have the same sign of $\Delta_d$. As in Section 1, we denote by $\operatorname{EM}^{(d)}(k)$ the least $N$ such that every $N$-term planar sequence of points in $d$-general position contains a $d$-th order monotone subsequence of size $k$. Thus first-order monotonicity is ordinary monotonicity, and second-order monotonicity is the cup/cap condition from \cite{ES35}. The first genuinely new case is $d=3$. 

As already alluded above, the central structural fact is that these sign colorings are transitive \cite[Lemma 2.5]{EM}. This means that $\operatorname{EM}^{(d)}(k)\le R_{d+1}^{\mathrm{trans}}(k)$, and so \eqref{EMblah} gives $\operatorname{EM}^{(d)}(k)\le \twr_d(O(k))$ for every fixed $d\ge 2$. For $d=3$, Eli\'a\v{s}--Matou\v{s}ek \cite[Theorem 1.5]{EM} also showed that there is an absolute constant $c>0$ such that
\begin{equation} \label{EMconstr}
\operatorname{EM}^{(3)}(k)\ge 2^{2^{ck}}
\end{equation}
for all $k\ge 1$. Together with the upper bound $\operatorname{EM}^{(3)}(k) \le R_4^{\mathrm{trans}}(k)\le 2^{2^{O(k)}}$, this allowed Eli\'a\v{s}--Matou\v{s}ek to show that $\operatorname{EM}^{(3)}(k)=2^{2^{\Theta(k)}}$. 

We are using \eqref{EMconstr} as a black box in this paper, but as a final point it is perhaps worth highlighting why parabolas naturally appear in their construction. Eli\'a\v{s} and Matou\v{s}ek build the lower-bound example for \eqref{EMconstr} recursively. At each step, every point of the previous configuration is replaced by a tiny affine copy of that configuration, and these copies are then bent upward by adding a very steep quadratic function. This ``cluster plus parabola'' construction forces the sign of every third divided difference to be governed by the combinatorial type of the quadruple, which prevents large third-order monotone subsequences. After $m$ steps, the resulting examples have size $2^{2^{m-1}}$ and contain no third-order monotone subsequence of size $2m+1$.

\section{Proof of Theorem~\ref{thm:higher-dimensional}}

For convenience, for a $(d+1)$-sequence $x_0,\ldots,x_d\in\mathbb R^d$ with
cyclic projections, we say its above--below color is \emph{positive} if
$(-1)^d(H_{\mathrm{even}}-H_{\mathrm{odd}})>0$, and \emph{negative} otherwise.

The next lemma gives a more practical way to think about the above--below
coloring.

\begin{lemma}\label{lem:color-det}
For a sequence $x_0,\ldots,x_d\in\mathbb R^d$ with cyclic projections, its
above--below color is positive if and only if
\[
\det
\begin{pmatrix}
1&1&\cdots&1\\
x_0&x_1&\cdots&x_d
\end{pmatrix}
>0.
\]
\end{lemma}

\begin{proof}
Write $x_i=(z_i,h_i)$, where $z_i\in\mathbb R^{d-1}$ and $h_i\in\mathbb R$.
Since the projections have cyclic order type, the tuple $z_0,\ldots,z_d$ has
the alternating Radon partition. Let $\rho$ be the Radon point, and write
\[
\rho=\sum_{j\text{ even}}\lambda_jz_j
=
\sum_{j\text{ odd}}\lambda_jz_j,
\qquad
\sum_{j\text{ even}}\lambda_j=
\sum_{j\text{ odd}}\lambda_j=1,
\]
with all $\lambda_j>0$. Then
\[
\sum_{j=0}^d(-1)^j\lambda_j=0,
\qquad
\sum_{j=0}^d(-1)^j\lambda_jz_j=0,
\qquad
\sum_{j=0}^d(-1)^j\lambda_jh_j=H_{\mathrm{even}}-H_{\mathrm{odd}}.
\]
Let
\[
D=
\det
\begin{pmatrix}
1&1&\cdots&1\\
z_0&z_1&\cdots&z_d\\
h_0&h_1&\cdots&h_d
\end{pmatrix}.
\]
By replacing the first column by the linear combination
\[
\frac{1}{\lambda_0}\sum_{j=0}^d(-1)^j\lambda_j
\begin{pmatrix}1\\ z_j\\ h_j\end{pmatrix},
\]
which does not change the determinant except for the factor $1/\lambda_0$, we get
\[
D=
\frac{(-1)^d}{\lambda_0}(H_{\mathrm{even}}-H_{\mathrm{odd}})
\det
\begin{pmatrix}
1&\cdots&1\\
z_1&\cdots&z_d
\end{pmatrix}.
\]
The last determinant is positive by the cyclic projection convention. Hence
$D>0$ if and only if $(-1)^d(H_{\mathrm{even}}-H_{\mathrm{odd}})>0$.
\end{proof}

We also need the following elementary linear algebra facts.

\begin{lemma}\label{lem:cramer}
Let $P$ be an $n\times(n+1)$ matrix. For $0\le j\le n$, let $P_j$ denote the
$n\times n$ matrix obtained by deleting the $(j+1)$-st column of $P$. Then
\[
\bigl(\det P_0,-\det P_1,\ldots,(-1)^n\det P_n\bigr)^T\in \ker P.
\]
\end{lemma}

\begin{proof}
If $y_i^T$ denotes the $i$-th row of $P$, then expanding the determinant of the
matrix obtained from $P$ by repeating the row $y_i^T$ at the top gives
\[
0=
\det
\begin{pmatrix}
y_i^T\\ P
\end{pmatrix}
=
y_i^T\cdot
\bigl(\det P_0,-\det P_1,\ldots,(-1)^n\det P_n\bigr)^T.
\]
This holds for every row $y_i^T$, so the vector lies in $\ker P$.
\end{proof}

The next lemma is a high-dimensional generalization of the classical theorem of Ptolemy from Euclidean geometry.

\begin{lemma}\label{lem:Plucker}
Let \(Q\) be an \(n\times(n+2)\) matrix, with columns $v_0,v_1,\ldots,v_{n+1}$, where \(n\geq 2\). For \(0\leq i<j\leq n+1\), let $\delta_{i,j}
  =
  \det(v_0,\ldots,\widehat v_i,\ldots,\widehat v_j,\ldots,v_{n+1})$, where the remaining \(n\) columns are kept in their original order. Then, for every $0\leq i_1<i_2<i_3<i_4\leq n+1$, we have
\begin{equation}\label{eq:plucker}
  \delta_{i_1,i_2}\delta_{i_3,i_4}
  -
  \delta_{i_1,i_3}\delta_{i_2,i_4}
  +
  \delta_{i_1,i_4}\delta_{i_2,i_3}
  =
  0.
\end{equation}
\end{lemma}

In modern language, Lemma \ref{lem:Plucker} can be regarded as special case of the usual Pl\"ucker relations for maximal
minors. A good reference is, for instance, Bruns and Vetter~\cite[Lemma 4.4]{BrunsVetter}. For the sake of self-containment, we include a direct proof for this particular case. 

\begin{proof}
It is enough to prove the case $(i_1,i_2,i_3,i_4)=(0,1,2,3)$, since the general case follows by permuting the columns of \(Q\). The signs in
\eqref{eq:plucker} are preserved by this relabeling because the complementary
minors are always taken with the induced column order.

Consider the following \(2n\times 2n\) matrix:
\[
M=
\begin{pmatrix}
v_4&\cdots&v_{n+1}&v_0&v_1&v_2&v_3&0&\cdots&0\\
0&\cdots&0&v_0&v_1&v_2&v_3&v_4&\cdots&v_{n+1}
\end{pmatrix}.
\]
Here the first and last blocks each have \(n-2\) columns; when \(n=2\), these
blocks are empty.

We first compute \(\det M\) by Laplace expansion along the first \(n\) rows. A
nonzero term must choose the first \(n-2\) columns, together with exactly two of
the four middle columns \(v_0,v_1,v_2,v_3\). The complementary minor from the
last \(n\) rows then contains the other two middle columns, together with the
last \(n-2\) columns. Collecting the six possible choices, and keeping track of
the Laplace signs, gives
\[
  \det M
  =
  2\bigl(
  \delta_{0,1}\delta_{2,3}
  -
  \delta_{0,2}\delta_{1,3}
  +
  \delta_{0,3}\delta_{1,2}
  \bigr).
\]

On the other hand, add each of the last \(n-2\) columns of \(M\) to the
corresponding one of the first \(n-2\) columns. Then replace each of the last
\(n\) rows by itself minus the corresponding one of the first \(n\) rows. These
elementary operations do not change the determinant. After them, the lower-left
\(n\times(n+2)\) block is zero, so the last \(n\) rows have nonzero entries only
in the last \(n-2\) columns. Hence the resulting matrix has rank strictly less
than \(2n\), and therefore $\det M=0$. This proves the case \((0,1,2,3)\), and hence the lemma.
\end{proof}

\begin{remark}
For $n=2$, Lemma~\ref{lem:Plucker} is also the standard Pl\"ucker relation for the Grassmannian
\(\operatorname{Gr}(2,4)\):
\[
\det(v_0,v_2)\det(v_1,v_3)
=
\det(v_0,v_1)\det(v_2,v_3)
+
\det(v_0,v_3)\det(v_1,v_2).
\]
If \(A_j=e^{i\theta_j}\) are four cyclically ordered points on the unit circle and $v_j=\bigl(\cos(\theta_j/2),\sin(\theta_j/2)\bigr)$, then
\[
  \det(v_i,v_j)
  =
  \sin\left(\frac{\theta_j-\theta_i}{2}\right)
  =
  \frac12 |A_iA_j|
  \qquad \text{for all}\ 0 \leq i<j \leq 3.
\]
Thus the above Pl\"ucker relation becomes the classical Ptolemy's theorem
\[
  |A_0A_2|\,|A_1A_3|
  =
  |A_0A_1|\,|A_2A_3|
  +
  |A_0A_3|\,|A_1A_2|.
\]
\end{remark}

We are now ready to complete the proof of Theorem \ref{thm:higher-dimensional}.

\begin{proof}[Proof of Theorem~\ref{thm:higher-dimensional}]
We first prove that the above--below coloring is $(d+1)$-monotone. It is enough to prove the one-switch property on every ordered $(d+2)$-tuple. Let $x_0,x_1,\ldots,x_{d+1}\in\mathbb R^d$ have cyclic projections, and write $x_i=(z_i,h_i)$, where
$z_i\in\mathbb R^{d-1}$. Put
\[
P=
\begin{pmatrix}
1&1&\cdots&1\\
x_0&x_1&\cdots&x_{d+1}
\end{pmatrix},
\qquad
Q=
\begin{pmatrix}
1&1&\cdots&1\\
z_0&z_1&\cdots&z_{d+1}
\end{pmatrix}.
\]
Thus $P$ is a $(d+1)\times(d+2)$ matrix and $Q$ is the $d\times(d+2)$ matrix
obtained from $P$ by deleting the last row.

For $0\le j\le d+1$, let $D_j$ be the determinant of the
$(d+1)\times(d+1)$ matrix obtained from $P$ by deleting the $(j+1)$-st column.
By Lemma~\ref{lem:color-det}, the above--below color of the $(d+1)$-tuple $\{x_0,\ldots,x_{d+1}\}\setminus\{x_j\}$ is the sign of $D_j$.

For $0\le a<b\le d+1$, let $\delta_{a,b}$ be the determinant of the
$d\times d$ matrix obtained from $Q$ by deleting the $(a+1)$-st and $(b+1)$-st
columns. By the cyclic projection condition, and with the same orientation
convention as above, all these minors are positive: $\delta_{a,b}>0$.

By Lemma~\ref{lem:cramer}, applied to $P$, the vector $\vec v_0=(D_0,-D_1,\ldots,(-1)^dD_d,(-1)^{d+1}D_{d+1})^T$ lies in $\ker P$, and hence also in $\ker Q$. Applying Lemma~\ref{lem:cramer}
to the first $d+1$ columns of $Q$ and to the last $d+1$ columns of $Q$, and
then extending by a zero coordinate, gives
\[
\vec v_1=(\delta_{0,d+1},-\delta_{1,d+1},\ldots,(-1)^d\delta_{d,d+1},0)^T,\ \ \vec v_2=(0,-\delta_{0,1},\delta_{0,2},\ldots,(-1)^{d+1}\delta_{0,d+1})^T
\in\ker Q.
\]
Since $Q$ has rank $d$, its kernel is two-dimensional, and $\vec v_1,\vec v_2$
form a basis. Hence $\vec v_0=\alpha\vec v_1+\beta\vec v_2$ for some real
numbers $\alpha,\beta$. Comparing the first and last coordinates gives
\[
\alpha=\frac{D_0}{\delta_{0,d+1}},
\qquad
\beta=\frac{D_{d+1}}{\delta_{0,d+1}}.
\]
Therefore, for every $1\le j\le d$,
\begin{equation}\label{eq:Dj-monotone-combination}
D_j
=
\frac{D_0\,\delta_{j,d+1}+D_{d+1}\,\delta_{0,j}}
     {\delta_{0,d+1}}.
\end{equation}

We now use the Pl\"ucker relation from Lemma \ref{lem:Plucker} for $(i_1,i_2,i_3,i_4)=(0,j,j+1,d+1)$. We get
\begin{equation}\label{eq:plucker-monotone-ratio}
\delta_{0,j+1}\delta_{j,d+1}
=
\delta_{0,j}\delta_{j+1,d+1}
+
\delta_{0,d+1}\delta_{j,j+1},
\qquad 1\le j\le d-1.
\end{equation}
Since every term in \eqref{eq:plucker-monotone-ratio} is positive, note that this implies that the ratios
\[
\rho_j:=\frac{\delta_{j,d+1}}{\delta_{0,j}},
\qquad 1\le j\le d,
\]
are strictly decreasing. This can be leveraged as follows. If $D_0$ and $D_{d+1}$ have the same sign, then
\eqref{eq:Dj-monotone-combination} implies that every $D_j$, $1\le j\le d$,
has this same sign. Suppose instead that $D_0$ and $D_{d+1}$ have opposite
signs. If $D_0>0$ and $D_{d+1}<0$, then
\eqref{eq:Dj-monotone-combination} gives
\[
D_j>0
\quad\Longleftrightarrow\quad
\rho_j>-\frac{D_{d+1}}{D_0}.
\]
However, we know that $\rho_1>\rho_2>\cdots>\rho_d$, thereby the signs of $D_0,D_1,\ldots,D_d,D_{d+1}$ change at most once. The case $D_0<0$ and $D_{d+1}>0$ is identical, with the
inequality reversed.

Thus the signs of the $D_j$'s form a one-switch sequence in deletion order.
The lexicographic order on the $(d+1)$-subsets of
$\{0,1,\ldots,d+1\}$ is the reverse of the deletion order
$j=0,1,\ldots,d+1$, and reversing a sequence preserves the property of having
at most one sign change. Hence the above--below coloring is
$(d+1)$-monotone. Consequently, $\operatorname{AB}^{(d)}(k)\le \overline R_{\mathrm{mon}}(k;d+1)$.

It remains to prove the moment-curve identity $\operatorname{AB}^{(d)}_{\mathrm{mc}}(k)=\operatorname{EM}^{(d)}(k)$. The main idea is the following generalized Vandermonde identity.

\begin{claim}\label{claim:vandermonde-divdiff}
Let $p_0=(t_0,h_0),\dots,p_d=(t_d,h_d)$, where $t_0<\cdots<t_d$. Then
\[
\det
\begin{pmatrix}
1&1&\cdots&1\\
t_0&t_1&\cdots&t_d\\
t_0^2&t_1^2&\cdots&t_d^2\\
\vdots&\vdots&&\vdots\\
t_0^{d-1}&t_1^{d-1}&\cdots&t_d^{d-1}\\
h_0&h_1&\cdots&h_d
\end{pmatrix}
=
\left(\prod_{0\le i<j\le d}(t_j-t_i)\right)
\Delta_d(p_0,\dots,p_d).
\]
\end{claim}

While this is likely fairly standard as well, we would like to include an argument which we think best explains the connection between higher order divided differences and cyclic projections (this is also the starting point of this paper, in some sense). Let $f(x)=a_0+a_1x+\cdots+a_dx^d$ be the unique polynomial of degree at most $d$ satisfying $f(t_i)=h_i$ for all
$i=0,\dots,d$. Using the Lagrange interpolation formula, it is easy to see that
the leading coefficient $a_d$ is precisely the divided difference $a_d=\Delta_d(p_0,\dots,p_d)$. On the other hand, in the determinant on the left hand side of
Claim~\ref{claim:vandermonde-divdiff}, the last row is
\[
(h_0,\dots,h_d)
=
a_0(1,\dots,1)+a_1(t_0,\dots,t_d)+\cdots+a_d(t_0^d,\dots,t_d^d).
\]
Subtracting from the last row the corresponding linear combination of the
previous rows leaves $a_d(t_0^d,\dots,t_d^d)$ as the last row. Therefore the
determinant becomes
\[
a_d
\det
\begin{pmatrix}
1&1&\cdots&1\\
t_0&t_1&\cdots&t_d\\
t_0^2&t_1^2&\cdots&t_d^2\\
\vdots&\vdots&&\vdots\\
t_0^{d-1}&t_1^{d-1}&\cdots&t_d^{d-1}\\
t_0^d&t_1^d&\cdots&t_d^d
\end{pmatrix}.
\]
The remaining determinant is the usual Vandermonde determinant,
$\prod_{0\le i<j\le d}(t_j-t_i)$, and since
$a_d=\Delta_d(p_0,\dots,p_d)$, this proves Claim~\ref{claim:vandermonde-divdiff}.

Since the Vandermonde factor $\prod_{0\le i<j\le d}(t_j-t_i)$ is positive,
Lemma~\ref{lem:color-det} shows that the above--below color of
$x_0,\ldots,x_d$ in the moment-curve model is determined, up to the fixed sign
convention in the definition of positivity, by the sign of
$\Delta_d(p_0,\ldots,p_d)$, and vice versa. Hence monochromatic subsequences in
the moment-curve model are exactly $d$-th order monotone subsequences. Thus $\operatorname{AB}^{(d)}_{\mathrm{mc}}(k)=\operatorname{EM}^{(d)}(k)$. Since the moment-curve model is a special case of cyclic projections, we also
have $\operatorname{AB}^{(d)}_{\mathrm{mc}}(k)\le \operatorname{AB}^{(d)}(k)$. Together with the monotone-coloring bound proved above, this completes the proof of Theorem~\ref{thm:higher-dimensional}.
\end{proof}

\subsection*{Acknowledgments}
  CP would like to thank Zichao Dong for useful discussions during his recent visit at Emory (in particular, for explaining to him the Eli\'a\v{s}--Matou\v{s}ek construction). 

We would also like to acknowledge the role of AI in preparing and enriching this manuscript. The upper bound $\operatorname{AB}^{(d)}(k)
\le \overline R_{\mathrm{mon}}(k;d+1)$ from Theorem \ref{thm:higher-dimensional} is entirely due to ChatGPT. This research was supported by NSF grant DMS-2246659.

\end{document}